\documentclass{my}

\newcommand{\R}{\mathbb{R}}
\newcommand{\N}{\mathbb N}
\newcommand{\Z}{\mathbb Z}

\newcommand{\C}{\mathbb{C}}

\newcommand{\cK}{\mathcal{K}}
\newcommand{\cB}{\mathcal{B}}

\newcommand{\cV}{\mathcal{V}}
\newcommand{\cP}{\mathcal{P}}
\newcommand{\cU}{\mathcal{U}}

\newcommand{\wt}{\widetilde}
\newcommand{\wb}[1]{\overline{#1}}
\newcommand{\eps}{\varepsilon}

\DeclareMathOperator{\Prop}{prop}
\theoremstyle{plain}
\newtheorem*{mainthmrestated}{Theorem \ref{thm_main_theorem} (precise form)}
\DeclareMathOperator{\spec}{spec}

\renewcommand{\Re}{\operatorname{Re}}

\newcommand{\cE}{\mathcal{E}}

\title{Quantitative \textit{K}-homology of uniform Roe algebras}
\author{Hao Guo, Chenxi Wang, Guoliang Yu}
\begin{document}
\maketitle
\begin{abstract}
We introduce quantitative $K$-homology for uniform Roe algebras of discrete proper metric spaces with bounded geometry, based on the controlled $KK$-theory of Willett and Yu. We construct a comparison map from Lipschitz $K$-theory and prove that, for chain geodesic spaces, it is an isomorphism in a quantitative sense; in particular, quantitative $K$-homology is isomorphic to the inverse limit of the Lipschitz $K$-theory groups as the Lipschitz constant tends to zero. We also establish a pairing between quantitative $K$-homology and quantitative $K$-theory. 
\end{abstract}

\tableofcontents

\section{Introduction}

In this paper, we formulate and study a notion of quantitative $K$-homology for uniform Roe algebras. Our definition is modelled on the controlled $KK$-theory of Willett and Yu \cite{WYUCT_control_kk2,controkk1}, as well as on earlier work describing $KK$-theory in terms of localization algebras of separable $C^*$-algebras \cite{DadarlatWillettWu2018Localization,Yu1997Localization}.

Our goal is to relate quantitative $K$-homology to Lipschitz $K$-theory of the underlying space. Lipschitz $K$-theory was defined in \cite{Wang_Xie_Yu_Decay_of_scalar_LipK}, using Lipschitz control on the representatives of $K$-theory classes and equivalence relations, and applied to index pairings to obtain decay estimates for scalar curvature. Subsequently it was shown in \cite{GuoYu2023Covering} that there is a natural pairing between Lipschitz $K$-theory and the quantitative $K$-theory of Roe algebras. Given the natural duality between $K$-theory and $K$-homology, that work suggests that there should be a controlled version of $K$-homology, and that it should be closely related to Lipschitz $K$-theory. In this paper, we introduce such a theory and prove that, for chain geodesic spaces, the two are isomorphic in a quantitative sense.

We first give an idea of the construction, with the precise definitions being provided in Section \ref{sec2}. Let $X$ be a discrete proper metric space with bounded geometry. The uniform Roe algebra $C_u^*(X)$ is the norm closure of the collection of all finite propagation operators in $\cB(\ell^2(X))$. For $r>0$, choose a subset $F$ of the unit ball of $C_u^*(X)$ consisting of operators of propagation at most $r$. For $\eps>0$, the group $K^*_{\eps}(F)$ is defined through equivalence classes of pairs of projections (resp. unitaries) in $M_n(\cK(\ell^2(X))^+)$  that almost commute with elements of $F$ up to $\eps$. We call $K^*_{\eps}(F)$ the \emph{local $K$-homology} of $C_u^*(X)$ with respect to $F$. It is an easy consequence of the definition that shrinking $\eps$ or enlarging $F$ gives natural forget-control maps between these groups. We then define the \emph{quantitative $K$-homology} $QK^*_r(X)$ to be the inverse limit of this system parameterized by such pairs $(F,\eps)$.

This group is naturally the target of a homomorphism from Lipschitz $K$-theory. In brief, a Lipschitz representative of a $K$-theory class acts on $\ell^2(X)$ by multiplication, and the Lipschitz estimate leads to commutator estimates with elements of $F$. This gives, for every $r>0$, comparison maps
\[
\nu\colon  K^{L}_*(C_0(X))\to K^*_{\Omega_X(r,L)}(F),\qquad \Omega_X(r,L)=4rL\sqrt{N_X(r)},
\]
where $N_X(r):=\sup_{x\in X}|B(x,r)|<\infty$ is the supremum of the cardinalities of the closed balls of radius $r$. These maps are compatible with the forget-control maps, which loosen the control on $F$ and on $L$.

In the case that $F=E_r$ is the entire unit ball of operators of propagation at most $r$, we will be interested in a notion of quantitative isomorphism that consists of the following two properties, stated here informally:
\begin{enumerate}[label=(\arabic*)]
    \item (Quantitative surjectivity) Every class in the local $K$-homology group $K^*_{\Omega_X(r,L)}(E_r)$ that lies in the forget-control image from a smaller scale $\Omega_X(r,L')$ is the image of a Lipschitz $K$-theory class. 
    \item (Quantitative injectivity) If the comparison image of a Lipschitz $K$-theory class vanishes, then the class itself vanishes after the Lipschitz constant is increased. 
\end{enumerate}
Our main theorem can be stated as follows; a precise version, with explicit constants, is given in Section 3.
\begin{theorem}\label{thm_main_theorem}
\label{thm main}
    Let $X$ be a discrete proper metric space with bounded geometry, and suppose that $X$ is $(m,C)$-chain geodesic. Let $r\ge \max(m,1)$ and $L> 0$ be such that $4rL\sqrt{N_X(r)}<1/100$. Then the comparison maps from Lipschitz $K$-theory to the local $K$-homology groups $K^*_{\Omega_X(r,L)}(E_r)$ of the uniform Roe algebra are quantitative isomorphisms. 
\end{theorem}
A discrete metric space is $(m,C)$-chain geodesic, where $m>0$, $C\ge 1$ are constants, if any two distinct points $x,y$ can be connected by a sequence of at most $C\cdot d(x,y)$ steps, each step having length at most $m$.

Passing to inverse limits, Theorem \ref{thm_main_theorem} yields an isomorphism $QK^*_r(X)\cong\varprojlim_{L\to 0}K^L_*(C_0(X))$ for every $r\ge\max(m,1)$ (Corollary \ref{cor_inverse_limit}); in particular, $QK^*_r(X)$ does not depend on $r$ in this range, and $QK^*(X)$ is isomorphic to the same inverse limit.

As explained above, the comparison map $\nu$ is defined in a natural manner using the Lipschitz condition on $K$-theory. The main work in proving Theorem \ref{thm main} is to go in the opposite direction, from a cycle for $K^*_\eps(E_r)$ back to a Lipschitz function. For this we use two families of test operators in $E_r$. On the one hand, approximate commutation with all diagonal sign operators forces a projection or unitary in $M_n(\cK(\ell^2(X))^+)$ to be close to a diagonal one, that is, to a matrix-valued function on $X$. On the other hand, approximate commutation with the shifts $\delta_x\mapsto\delta_y$, for $0<d(x,y)\le m$, bounds the variation of this function between nearby points. The chain geodesic condition then turns this local bound into a global Lipschitz bound. Applied to a single cycle, this procedure yields quantitative surjectivity; applied along a homotopy, it gives quantitative injectivity.

The paper is organized as follows. In Section 2, we define quantitative $K$-homology groups. In Section 3, we recall the definition of Lipschitz $K$-theory and construct the comparison maps alluded to above, before proving the main theorem in Section 4. In Section 5, we construct a pairing between the local $K$-homology groups and the quantitative $K$-theory groups, with values in $\Z$. We then compare our theory with the controlled $KK$-theory of Willett and Yu \cite{WYUCT_control_kk2}, defined using finite test sets, and conjecture that the canonical map from $QK^*(X)$ to the corresponding inverse limit in that theory is an isomorphism.

\subsection{Notation and standard notions}

We denote the unitalization of a $C^*$-algebra $A$ by $A^+$, which is distinct from $A$ even if $A$ is already unital (as happens for $\cK(\ell^2(X))$ when $X$ is finite), so that the canonical quotient map $\sigma$ from $M_n(A^+)$ to $M_n(\C)$ is always defined. 
The symbol $\otimes$ always denotes the completed tensor product of Hilbert spaces or the minimal tensor product of $C^*$-algebras. 

By a representation of $A$ we shall mean a nondegenerate representation of $A$ on a Hilbert space $H$, that is, a $*$-homomorphism $\pi\colon A\to \cB(H)$ such that $\wb{\pi(A)H}=H$. For $a\in A$, we usually write simply $a\in \cB(H)$ for $\pi(a)$ to avoid clutter. 

Let $X$ be a discrete proper metric space with bounded geometry. Let $C_0(X)$ denote the continuous functions on $X$ vanishing at infinity, and $\ell^2(X)$ the square-summable functions on $X$. Fix the representation $\rho_X\colon C_0(X)\to \cB(\ell^2(X))$ given by pointwise multiplication. For an operator $T\in\cB(\ell^2(X))$, the support of $T$ is defined to be the complement of the set of all points $(x,y)\in X\times X$ for which there exist functions $f,g\in C_0(X)$ such that $gTf=0$ and $f(x)\ne 0$, $g(y)\neq 0$. The propagation of $T$ is defined to be
\[ 
\Prop(T):=\sup\Set{d(x,y)\mid (x,y)\in \supp(T)}, 
\]
and an operator $T\in\cB(\ell^2(X))$ is said to be locally compact if $fT$ and $Tf$ are compact for all $f\in C_0(X)$. 

Let $\C_u^*[X]_r\subset \cB(\ell^2(X))$ be the subspace of bounded operators with propagation at most $r$. Viewed as a collection of matrices in the obvious way, $\C_u^*[X]_r$ consists of all $\C$-valued $X$-by-$X$ indexed matrices $(T_{xy})_{x,y\in X}$ such that 
\begin{itemize}
    \item[(1)] there exists $M>0$ such that $|T_{xy}|\le M$ for all $x,y\in X$;
    \item[(2)] if $d(x,y)>r$ then $T_{xy}=0$. 
\end{itemize}

The operator norm closure of the union $\bigcup_{r\geq 0}\C_u^*[X]_r$ is the uniform Roe algebra $C^*_u(X)$. The inclusion $C^*_u(X)\subset \cB(\ell^2(X))$ induces a diagonal embedding $a\mapsto a\otimes 1_n$ of $C^*_u(X)$ into $M_n(\cB(\ell^2(X)))$, which we denote by $\pi_X$, or simply $\pi$ when there is no ambiguity. We define $E_r$ to be the unit ball of $\C_u^*[X]_r$. 

For convenience, we let $e_l$ denote the rank $l$ projection in $M_n(\C)$ with $l$ ones in the top-left part of the diagonal and zeros elsewhere. If $l=n$, we write $1_n$ instead to emphasize that it is the identity matrix.

\section{Quantitative \textit{K}-homology of the uniform Roe algebra}
\label{sec2}

In this section, we define quantitative $K$-homology in the case of
uniform Roe algebras and establish some basic results on representatives
and controlled homotopies. Throughout, we set $H=\ell^2(X)$ and use the diagonal inclusion $\pi\colon C^*_u(X)\to M_n(\cB(H))$; we will omit $\pi$ from the notation when there is no ambiguity. As before, $\C_u^*[X]_r$ denotes the subspace of $C_u^*(X)$ consisting of operators of propagation at most $r$, and $E_r$ denotes the unit ball of $\C_u^*[X]_r$.

\subsection{Definitions and inverse limits}
We begin by defining the local groups $K_\varepsilon^*(F)$. Here $F$ may be an arbitrary subset of the unit ball of $C_u^*(X)$; the propagation filtration only enters in Definition \ref{def_quant_K_hom}. 

\begin{define}\label{def_K_hom_K0}
    Let $F$ be a subset of the unit ball of $C_u^*(X)$, and let $\varepsilon>0$. For each positive integer $n$, let $\cP_{n,\varepsilon}(F)$ be the collection of pairs $(p,q)$ of projections in $M_n(\cK(H)^+)$ such that: 
    \begin{enumerate}[label=(\arabic*)]
        \item $\| [p,a]\|<\varepsilon$ and $\|[q,a]\|<\varepsilon$,  $\forall a\in F$;
        \item $[\sigma(p)]=[\sigma(q)]$ in $K_0(\C)$, where $\sigma$ is the canonical quotient map from $M_n(\cK(H)^+)$ to $M_n(\C)$.
    \end{enumerate}
    Let $\cP_{\infty,\varepsilon}(F)$ denote the disjoint union of the $\cP_{n,\varepsilon}(F)$ over all $n$. Let $\sim$ be the equivalence relation on $\cP_{\infty,\varepsilon}(F)$ generated by declaring 
    $(p,q)\sim (p\oplus p',q\oplus p')$, and $(p_1,q_1)\sim (p_2,q_2)$ if the two pairs are connected by a path of projection pairs in $\cP_{\infty,\varepsilon}(F)$.

    We call the resulting set of equivalence classes the \textbf{local $K$-homology} of $C^*_u(X)$ with respect to $F$, and denote it by 
    $$
    K^0_{\varepsilon}(F)\coloneqq\cP_{\infty,\varepsilon}(F)/\sim.
    $$
On this set, define the binary operation 
        \[ [p_1,q_1]+[p_2,q_2]\coloneqq[p_1\oplus p_2, q_1\oplus q_2]. \]
\end{define}

  In degree one, we use unitaries instead of pairs of projections.
  \begin{define}\label{def_K_hom_K1}
      Let $F$ be a subset of the unit ball of $C_u^*(X)$, and let $\varepsilon>0$. For each positive integer $n$, let $\cU_{n,\eps}(F)$ be the collection of unitaries $u$ in $M_n(\cK(H)^+)$ such that 
        $\| [u,a]\|<\varepsilon$, $\|[u^*,a]\|<\eps$ for all $a\in F$.
        
    Let $\cU_{\infty,\eps}(F)$ denote the disjoint union of the $\cU_{n,\eps}(F)$ over all $n$. Let $\sim$ be the equivalence relation on $\cU_{\infty,\eps}(F)$ generated by declaring 
    $u\sim u\oplus 1_l$, and $u_1\sim u_2$ if they are connected by a path of unitaries in $\cU_{\infty,2\varepsilon}(F)$.

   The resulting set of equivalence classes is denoted by 
    $$
    K^1_{\varepsilon}(F)\coloneqq\cU_{\infty,\varepsilon}(F)/\sim.
    $$
On this set, define the binary operation 
        \[ [u_1]+[u_2]\coloneqq[u_1\oplus u_2]. \]
\end{define}

\begin{lemma}\label{lem_local_Khom_group}
    $K^0_{\varepsilon}(F)$ is a well-defined abelian group with identity element $[0,0]$, and $K^1_{\varepsilon}(F)$ is a well-defined abelian group whose identity element is the class $[1]$ of the identity operator in $\cB(H)$.
\end{lemma}
\begin{proof}
     It is clear that $K_{\varepsilon}^0(F)$ is a well-defined monoid with identity $[0,0]$. To see that it is abelian, let
     \[ 
     R(t)= \MAT{\cos(t), \sin(t); -\sin(t),\cos(t)},\quad t\in \Big [0,\frac{\pi}{2}\Big].
     \]
     Then $(p_1\oplus p_2,q_1\oplus q_2)\sim(p_2\oplus p_1, q_2\oplus q_1)$ via the rotation homotopy $R(t)(p_1\oplus p_2)R(t)^*$, after first enlarging $p_1,p_2$ by direct sums with $0$ so that they have the same size. Note that $[R(t),\pi(a)]=0$ because $R(t)$ is a scalar matrix.
     
     Next, we claim that every element $[p,q]$ has an inverse, namely $[q,p]$. Indeed, the homotopy 
     $$
     \left(\MAT{p,0;0,q}, R(t)\MAT{q,0;0,p}R(t)^* \right),\quad t\in\Big[0,\frac{\pi}{2}\Big]
     $$
     connects $(p\oplus q, q\oplus p)$ to $(p\oplus q,p\oplus q)\sim(0,0)$.

     Finally, the statements for $K_{\varepsilon}^1(F)$ follow from a similar argument, using the homotopy 
     $$\MAT{u,0;0,1} R(t)\MAT{1,0;0,u^*} R(t)^*,\quad t\in \Big[0,\frac{\pi}{2}\Big],$$
     of unitaries that almost commute with each $a\in F$ at level $2\eps$. This shows that the identity element is $[1]$ and that the inverse of $[u]$ is $[u^*]$.
\end{proof}

Next, we introduce forget-control maps, which relate the groups for different control parameters in the most general setting. Recall that for $I$ a directed set, an inverse system of abelian groups indexed by $I$ is a collection $(G_i)_{i\in I }$ of abelian groups, along with a group homomorphism $\phi_{ij}\colon G_j\rightarrow G_i$ for each $j\geq i$ such that $\phi_{ii}=\textnormal{id}$ and $\phi_{ij}\circ\phi_{jk}=\phi_{ik}$ whenever $k\geq j\geq i$. Given an inverse system, its inverse limit is defined to be 
\[
\lim_{\leftarrow I}G_i:=\left\{(g_i)\in \prod_{i\in I }G_i\mid \phi_{ij}(g_j)=g_i, \forall \ j\geq i\right\}.
\]
An element $(g_i)$ of this inverse limit is called a \textit{coherent family}. 

\begin{define}\label{def_forget_contrl_map}
    Let $\calX$ be the collection of all pairs $(F,\varepsilon)$, where $F$ is a subset of the unit ball of $C^*_u(X)$ and $\varepsilon>0$. Define a partial order on $\calX$ by $(F,\varepsilon)\leq(G,\delta) $ if:
    \begin{itemize}
        \item[(1)] $\varepsilon\geq \delta$;
        \item[(2)] for any $a\in F$ there exists $b\in G$ with $\|a-b\|\le \frac{\varepsilon-\delta}{2}$.
    \end{itemize} 
    One verifies that
    \[
    \cP_{n,\delta}(G) \subseteq \cP_{n,\varepsilon}(F) \text{ if }(F,\varepsilon)\leq(G,\delta),
    \]
    and hence there is a \emph{forget-control} map between the groups associated to the two pairs of control parameters,
    \[ 
    \phi^{(G,\delta)(F,\varepsilon) }\colon\  K^*_{\delta}(G)\rightarrow K^*_{\varepsilon}(F).
    \] 
    
    Equipped with these forget-control maps, the collection of pairs $(F,\eps)$ in $\calX$ with $F\subset \C_u^*[X]_r$ forms a directed subset $\calX_r$, and the groups indexed by it form an inverse system. When $F=G$, we simply write $\phi^{\delta,\eps}$. 
\end{define}
 
\begin{define}\label{def_quant_K_hom}
    Denote the inverse limit of the system $\calX_r$ by
    \begin{align*}
    QK^0_r(X)&:=\lim_{\eps\to 0}K^0_{\eps}(E_r)=\lim_{\leftarrow (F,\eps)\in \calX_r} K^0_{\varepsilon}(F),\\
    QK^1_r(X)&:=\lim_{\eps\to 0}K^1_{\eps}(E_r)=\lim_{\leftarrow(F,\eps)\in \calX_r} K^1_{\varepsilon}(F).
    \end{align*}
    Since $E_r\subseteq E_s$ if $r<s$, we also have the inverse limit 
    \[ QK^*(X):=\underset{r\to \infty}{\lim_{\longleftarrow}} QK^*_r(X).\]
    We refer to the groups $QK_r^*$ and $QK^*$ as the \textbf{quantitative $K$-homology} of the uniform Roe algebra $C^*_u(X)$.
\end{define}

\begin{remark}\label{rem_three_levels}
  The theory can thus be viewed at three levels. The local group $K^*_\eps(F)$ has a fixed test set $F$ and a fixed commutator tolerance; $QK^*_r(X)$ is the inverse limit as the tolerance tends to zero with the propagation bound $r$ fixed; and $QK^*(X)$ is the further inverse limit as $r\to\infty$. Our main theorem (Theorem \ref{thm_main_theorem}) is formulated at the first level, for $F=E_r$: it consists of surjectivity and injectivity statements with specified changes of control, which is the appropriate notion of isomorphism for the groups $K^*_\eps(E_r)$ regarded as an inverse system. 
\end{remark}

\begin{remark}
  When $F=\emptyset$ or $\eps>2$, the definition reduces to ordinary $K$-theory, namely $K_{\eps}^*(F)\cong K_*(\cK(H))$, which is $\Z$ in degree $0$ and $0$ in degree $1$. The case $r=0$ will be discussed in Example \ref{exam_2_pts_quant_K}. 
\end{remark}

\subsection{Homotopy paths from norm closeness}
In this subsection, we use almost projections and almost unitaries to produce homotopies, at the cost of increasing $\eps$. The standard result for dealing with almost projections and almost unitaries can also be found in \cite[Remark 1.4]{OyonoOyono2011}. 
\begin{prop}\label{prop_almost_proj_unitary}
    Let $A$ be a unital $C^*$-algebra, and let $\varepsilon<\frac{1}{4}$. 
    \begin{propparts}
        \item Suppose $p\in A$ satisfies $p=p^*$, $\|p^2-p\|<\varepsilon$, and let $f$ be a continuous function on $\R$ such that 
    \[ 
    f(t)= 0 \text{ if } t\le \frac{1-\sqrt{1-4\eps}}{2};\quad f(t)=1 \text{ if } t\ge \frac{1+\sqrt{1-4\eps}}{2}.
    \]
    Then $f(p)$ is a projection with $\|p-f(p)\|<2\eps$, and it does not depend on the choice of $f$.
    \item If $v$ is invertible in $A$ and satisfies $\|I-v^*v\|<\eps$, then the polar part\footnote{The definition of almost unitary used here is slightly different from, but equivalent to, the usual one, because $I-vv^*$ and $I-v^*v$ are both self-adjoint and their spectra can differ only at the point $1$. }
    $$
    u:=v(v^*v)^{-\frac{1}{2}}
    $$
    is a unitary such that $\|u-v\|<2\eps$.
    \end{propparts} 
\end{prop}
\begin{proof}
    Recall that for a self-adjoint element $a\in A$ and a continuous function $f\colon \R \to \R$, the continuous functional calculus $f(a)$ satisfies 
    \[ \|f(a)\|=\sup_{\lambda\in \spec(a)}|f(\lambda)|.   \]
    The condition $\|p^2-p\|<\eps$ implies that the spectrum of $p$ is contained in $\big[\frac{1-\sqrt{1+4\eps}}{2},\frac{1-\sqrt{1-4\eps}}{2}\big]\cup\big[\frac{1+\sqrt{1-4\eps}}{2},\frac{1+\sqrt{1+4\eps}}{2}\big]$. Hence the spectrum of $f(p)$ consists of $0,1$ only, and any two choices of $f$ give the same element, since a self-adjoint element with spectrum $\{0\}$ is zero. Moreover,  
    \[ 
    \|p-f(p)\|=\sup_{\lambda\in \spec(p)}|\lambda-f(\lambda)|\le\frac{1-\sqrt{1-4\eps}}{2}<2\eps.
    \]
    
    It is not hard to check that $u$ is a unitary, and that $\|I-v^*v\|<\eps$ implies $\|v\|\in\big[\sqrt{1-\eps},\sqrt{1+\eps}\big]$. Then 
    \begin{align*}
    \|u-v\|\le\|v\|\cdot \Big\|1-\big(1-(1-v^*v)\big)^{-\frac{1}{2}}\Big\|<2\eps.&\qedhere
    \end{align*}
\end{proof}

The above observations imply that projections or unitaries which are close in norm are homotopic.
\begin{cor}\label{cor_norm_closed_proj_uni}
    Let $\eps<1/4$. 
    \begin{propparts}
        \item If a projection pair $(p,q)\in \cP_{n,\eps}(F)$ satisfies $\|p-q\|<\eps$, then $(p,e_l),(q,e_l)$ are well-defined elements of $\cP_{n,2\eps}(F)$ and lie in the same path component. 
        \item If unitaries $u,v\in \cU_{n,\eps}(F)$ satisfy $\|u-v\|<\eps$, then $u,v$ lie in the same path component of $\cU_{n,2\eps}(F)$.
    \end{propparts}
\end{cor}
\begin{proof}
    Note that, by definition, $\sigma(p)$ and $\sigma(q)$ have the same rank, so we may take $[\sigma(p)]=[e_l]=[\sigma(q)]$. Define $p_t=tp+(1-t)q$; this is a path of $\eps^2/4$-almost projections, since $p$ and $q$ are projections and therefore
    \[ 
    \|p_t^2-p_t\|=(t-t^2)\|(p-q)^2\|\le\tfrac14\|p-q\|^2<\eps^2/4.   
    \]
    Moreover, $\|[f(p_t),a]\|\le \|[p_t-f(p_t),a]\|+\|[p_t,a]\|< 2\eps$. The path $f(p_t)$ is the desired one. 

    For unitaries, define $u_t=tu+(1-t)v$. This is a path of $\eps^2/4$-almost unitaries: the identity $u_t-u=(1-t)(v-u)$ implies that $u_t$ is invertible, and since $u$ and $v$ are unitaries, 
    \[ 
    \|I-u_t^*u_t\|=(t-t^2)\|(u-v)^*(u-v)\|\le\tfrac14\|u-v\|^2<\eps^2/4. 
    \]
    The path $u_t(u_t^*u_t)^{-\frac{1}{2}}$ is the desired one.
\end{proof}

\subsection{Normal forms}
Next, we state a lemma that allows us to prescribe the scalar part of cycles for $K^*_{\eps}(F)$. 
\begin{lemma}\label{lem_precribe_scalar_P1_U1}
    With notation as in Definitions \ref{def_K_hom_K0} and \ref{def_K_hom_K1}, define
    $$
    \cP^1_{n,\varepsilon}(F):=\Set{(p,q)\in \cP_{n,\varepsilon}(F)\mid \exists l\in \N, \,\sigma(p)=\sigma(q)=e_l},
    $$
    
    $$
    \cU^1_{n,\varepsilon}(F):=\Set{u\in \cU_{n,\varepsilon}(F)\mid u-1_n\in M_n(\cK(H))}.
    $$
    Then: 
    \begin{propparts}
        \item every element $(p,q)\in \cP_{n,\varepsilon}(F)$ lies in the same path component as some element of $\cP^1_{n,\varepsilon}(F)$;
        \item if two elements $(p_1,q_1)$ and $(p_2,q_2)$ of $\cP_{n,\varepsilon}^1(F)$ are connected by a path in $\cP_{n,\varepsilon}(F)$, then they are connected by a path in $\cP_{n,\varepsilon}^1(F)$; 
        \item every element $v\in \cU_{n,\varepsilon}(F)$ lies in the same path component as some element of $\cU^1_{n,\varepsilon}(F)$, and if $v_0,v_1\in\cU^1_{n,\varepsilon}(F)$ lie in the same path component of $\cU_{n,\varepsilon}(F)$, then they lie in the same path component of $\cU^1_{n,\varepsilon}(F)$.
    \end{propparts}
\end{lemma}
\begin{proof}
    For part (1), let $\sigma\colon M_n(\cK(H)^+)\to M_n(\C)$ be the canonical quotient map. Since $[\sigma(p)]=[\sigma(q)]\in K_0(\C)$ by definition, we may find a path of unitaries $(u_t)\in C([0,1],M_n(\C))$ such that $u_t \sigma(p)u_t^*$ connects $\sigma(p)$ to some $e_l$. Similarly, choose a path $(v_t)$ for $\sigma(q)$. Since scalar matrices commute with $F$, the path $(u_t pu_t^*,v_t qv_t^*)$ lies in $\cP_{n,\varepsilon}(F)$ and connects $(p,q)$ to an element of $\cP^1_{n,\varepsilon}(F)$.

    For part (2), suppose that $(p_0,q_0)$ and $(p_1,q_1)$ are elements of $\mathcal{P}^1_{n,\eps}(F)$ that are connected by a path $(p_t,q_t)_{t\in [0,1]}$ in $\mathcal{P}_{n,\eps}(F)$.  In particular, there exists $l\in \N$ such that $\sigma(p_0)=\sigma(q_0)=e_l=\sigma(p_1)=\sigma(q_1)$.
    By \cite[Lemma 4.7]{WYUCT_control_kk2}, there are paths $(u_t)_{t\in [0,1]}$ and $(v_t)_{t\in [0,1]}$ of unitaries in $M_n(\C)$ such that $\sigma(p_t)=u_t\sigma(p_0)u_t^*$ and $\sigma(q_t)=v_t\sigma(q_0)v_t^*$ for all $t\in [0,1]$.  The path $((u_t^*p_tu_t,v_t^*q_tv_t))_{t\in [0,1]}$ then lies in $\mathcal{P}^1_{n,\eps}(F)$ and connects $(p_0,q_0)$ to $(u_1^*p_1u_1,v_1^*q_1v_1)$.
    
    We claim that there is a path in $\mathcal{P}^1_{n,\eps}(F)$, parametrized by $[1,2]$, connecting $(u_1^*p_1u_1,v_1^*q_1v_1)$ and $(p_1,q_1)$.  Concatenating this path with the one above gives a path in $\mathcal{P}^1_{n,\eps}(F)$ connecting $(p_0,q_0)$ and $(p_1,q_1)$, which completes the proof.
    
    To establish the claim, note that $u_1$ commutes with $e_l$, and is therefore connected to the identity in $M_n(\C)$ via a path of unitaries that all commute with $e_l$, obtained by exponentiation; call this path $(u_t)_{t\in [1,2]}$.  Similarly, we get a path $(v_t)_{t\in [1,2]}$ with the same properties with respect to $v_1$.  The path $((u_t^*p_1u_t,v_t^*q_1v_t))_{t\in [1,2]}$ then lies in $\mathcal{P}^1_{n,\eps}(F)$ and connects $(u_1^*p_1u_1,v_1^*q_1v_1)$ to $(p_1,q_1)$, so we are done.

    For part (3), let $w=\sigma(v^*)$. Since the unitary group $U_n(\C)$ is connected, there is a path $(w_t)_{t\in [0,1]}$ of unitaries connecting the identity to $w$. Then the path $(w_tv)$ connects $v$ to $u=w_1 v$, which satisfies $\sigma(u)=1$, so that $1-u\in M_n(\cK(H))$. If $v_t$ is a path connecting $v_0$ and $v_1$, set $w_t=\sigma(v_t^*)$. Since $w_0=w_1=1$, the path $(w_tv_t)$ is the required one.
\end{proof}

Consequently, classes can be represented in a normalized form, as in ordinary $K$-theory. For the rest of the paper, we assume that $F$ is $*$-closed. 

\begin{lemma}\label{lem_element_to_[p,e_n]}
    Every element of $K^0_{\varepsilon}(F)$ can be written as $[p,1_n\oplus 0_n]$ for some $n$, with $\sigma(p)=1_n\oplus 0_n$.
\end{lemma}
\begin{proof}
    Given a class $[p,q]$ in $K^0_{\varepsilon}(F)$, we may assume that $(p,q)\in \cP^1_{n,\eps}(F)$. Note that 
    $$[p,q]=[p\oplus (1_n-q),q\oplus (1_n-q)]=[p\oplus (1_n-q),1_n\oplus 0_n]$$
    via the rotational path of projections almost commuting with $F$ at level $\varepsilon$,
    \[
    p_t=q\oplus 0_n+R(t)\big((1-q)\oplus 0_n\big)R(t)^*,\quad [p_t,a]=\MAT{\sin^2 t, \sin t\cos t;\cos t \sin t, -\sin^2t }[q,a].
    \]
     A further rotational homotopy on the first component modifies $\sigma(p\oplus (1_n-q))$ to $1_n\oplus 0_n$.
\end{proof}

Finally, we give a criterion for a class to be zero, which will be used in the discussion of injectivity. The proof is the same as that of \cite[Lemma 3.3]{WYUCT_control_kk2}. 

\begin{lemma}\label{lem_K0_element_to_be_zero}
With notation as in Definition \ref{def_K_hom_K0}, let $(p,q)\in \mathcal{P}_{n,\eps}(F)$, and assume that $[p,q]=0$ in $K^0_{\eps}(F)$.   Then there exist $m\in \N$ and an element $(s,s)$ of $\mathcal{P}_{n+2m,\eps}(F)$ such that $(p\oplus 1_m\oplus 0_m,q\oplus 1_m\oplus 0_m)$ lies in the same path component of $\mathcal{P}_{n+2m, \eps}(F)$ as $(s,s)$.
\end{lemma}

\begin{proof}
The condition $[p,q]=0\in K^0_{\eps}(F)$ implies that there are elements $(r,r)$ and $(x,x)$ in $\mathcal{P}_{\infty,\eps}(F)$ such that $(p\oplus r,q\oplus r)$ is homotopic to $(x,x)$. Note that $(p\oplus r \oplus (1-r),q\oplus r\oplus (1-r))$ is then homotopic to $(x\oplus (1-r),x\oplus (1-r))$. Suppose that $r\in M_l(\cK(H)^+)$. For $t\in [0,\pi/2]$, define  
$$
r_t:=\begin{pmatrix} r & 0 \\ 0 & 0 \end{pmatrix} +\begin{pmatrix} \cos(t) & -\sin(t) \\ \sin(t) & \cos(t)\end{pmatrix} \begin{pmatrix} 0 & 0 \\ 0 & 1-r \end{pmatrix}\begin{pmatrix} \cos(t) & \sin(t) \\ -\sin(t) & \cos(t)\end{pmatrix}.
$$
Then $(r_t)_{t\in [0,\pi/2]}$ is a path in $M_{2l}(\cK(H)^+)$ connecting $r\oplus (1-r)$ and $1\oplus 0$.  One verifies that $r_t$ is a projection for all $t$, and that $\|[r_t,a]\|<\eps$ for all $a\in F$.  Hence, with $m=l$ and $s=x\oplus (1-r)$, we obtain the claimed result.
\end{proof}

\section{Comparison maps}
In this section, we recall Lipschitz $K$-theory, construct the comparison maps to local $K$-homology, and state the main theorem precisely.

\subsection{Lipschitz \textit{K}-theory}
Definitions and basic properties of Lipschitz $K$-theory were introduced in \cite[Section 4]{Wang_Xie_Yu_Decay_of_scalar_LipK}. We review the case of continuous functions here and omit the proofs. For simplicity, we will use notation that differs slightly from the reference. 
\begin{define}\label{def_Lip_proj_uni}
	Let $X$ be a locally compact metric space, and let $L\ge 0$. An element $f\in M_n(C_0(X)^+)$ is said to be $L$-Lipschitz if 
    \[
    \|f(x)-f(y)\|_{M_n(\C)}\le L\,d(x,y), 
    \]
    for all $x,y\in X$, and the set of such elements is denoted by $M_n(C_0(X)^+)_L$. 
    \\ Let $P^L_n(X)$ be the set of projections in $M_n(C_0(X)^+)_L$, with inclusions given by
	$$
    P_n^L(X)\to P_{n+1}^L(X),\ p\mapsto \begin{pmatrix}
	p &0\\0&0
	\end{pmatrix},\quad P^L(X)=\lim P^L_n(X).
    $$
	Let $U_n^L(X)$ be the set of unitaries in $M_n(C_0(X)^+)_L$, with the natural inclusions given by
	$$
    U_n^L(X)\to U_{n+1}^L(X),\ u\mapsto 
    \begin{pmatrix}
	u &0\\0&1
	\end{pmatrix},\quad U^L(X)=\lim U^L_n(X).
    $$
\end{define} 
	
	We define the following equivalence relations.
	\begin{itemize}
		\item For $(p,l)$ and $(q,l')$ in $P^L(X)\times\N$, we say $(p,l)\sim (q,l')$ if $p\oplus I_{j+l'}$ and $q\oplus I_{j+l}$ are homotopic in $P^{2L}(X)$ for some $j$.
		\item For $u$ and $v$ in $U^L(X)$, we say $u\sim v$ if $u$ and $v$ are homotopic in $U^{2L}(X)$.
	\end{itemize}
	\begin{define}\label{def_LipK_groups}
		We denote by $\sigma$ the canonical homomorphism from $C_0(X)^+$ to $\C$. The Lipschitz controlled $K$-groups of $C_0(X)$ are defined as follows: 
			$$K_0^L(C_0(X)):=\{(p,l)\in P^L(X)\times \N: \mathrm{rank}(\sigma(p))=l\}/\sim.$$
			$$K_1^L(C_0(X)):=U^L(X)/\sim.$$
	        The groups $K_*^L(C_0(X))$ are abelian groups under the operations
	        $[p,l]+[p',l']=[p\oplus p',l+l']$
	        and
	        $[u]+[v]=[u\oplus v]$, respectively.
	\end{define}

    By the same rotation trick as in Lemma \ref{lem_element_to_[p,e_n]}, we have the following normalization.
    \begin{lemma}
    \label{lem_LipK_prescribe}
        Every class $z$ in $K_0^L(C_0(X))$ has a representative $[p,n]$ with $p\in P^L_{2n}(X)$ and $\sigma(p)=1_n\oplus 0_n$ for some $n$. Every class $z$ in $K_1^L(C_0(X))$ has a representative $[u]$ with $u\in U^L_{n}(X)$ and $\sigma(u)=1_n$ for some $n$.
    \end{lemma}
    \begin{proof}
    Let $(p,l)$ represent $z\in K_0^L(C_0(X))$, with $p\in P^L_n(X)$, so that $\mathrm{rank}(\sigma(p))=l$. The pair $(1_n-\sigma(p),n-l)$ represents zero, being a scalar projection of the correct rank, so $z=[p\oplus(1_n-\sigma(p)),n]$. The scalar part of $p\oplus(1_n-\sigma(p))$ is a projection of rank $n$ in $M_{2n}(\C)$, hence of the form $w^*(1_n\oplus 0_n)w$ for a scalar unitary $w$, which is connected to $1$ through scalar unitaries. Since conjugation by scalar unitaries preserves Lipschitz constants, $z=[w(p\oplus(1_n-\sigma(p)))w^*,n]$ is the required representative. In the odd case, $1_n$ is connected to $\sigma(u)^*\in U_n(\C)$ by a path $(w_t)$ of scalar unitaries, and $(w_tu)$ is a path in $U^L_n(X)$ from $u$ to $w_1u$, which satisfies $\sigma(w_1u)=1_n$.
    \end{proof}
    
    Since $C_0(X)^+_L\subset C_0(X)^+_{L'}$ if $L<L'$, there is a natural \textit{forget-control} map
$$ 
    i^{L,L'}\colon K_*^L(C_0(X))\longrightarrow K_*^{L'}(C_0(X)).
$$
   The collection of abelian groups $\{K_*^L(C_0(X))\}_{L\geqslant 0}$, together with the homomorphisms $\{i^{L,L'}\}_{0\leqslant L<L'}$, forms a direct system over $\R_{\geqslant 0}$ as $L$ increases. Restricted to $L>0$ and read in the direction of decreasing $L$, the same family is an inverse system. The inverse limit $\varprojlim_{L\to 0}K^L_*(C_0(X))$, the group of coherent families $(z_L)_{L>0}$ with $i^{L,L'}(z_L)=z_{L'}$ for $L\le L'$, appears in Corollary \ref{cor_inverse_limit} below. 
   
\subsection{The comparison maps}
Let $X$ be a discrete proper metric space with bounded geometry, let $H=\ell^2(X)$, and let $\rho_X\colon C_{0}(X)\rightarrow \cB(H)$ be the representation by multiplication. Let 
\[ 
N_X(r):=\sup_{x\in X}|B(x,r)|<\infty.
\]
The range of $\rho_X$ lies in $\cK(H)$. Therefore, we have a map
\begin{align*}
    \nu\colon  C_0(X)^+&\longrightarrow \cK(H)^+\\
    (f,\lambda) &\mapsto (\rho_X(f),\lambda),
\end{align*}
which we extend to $M_n(C_0(X)^+)\to M_n(\cK(H)^+)$, still denoted by $\nu$. For notational convenience, we may omit $\pi$ and $\rho$.

\begin{define}\label{def_local_map}
For any $r>0$, $L>0$ and any subset $F$ of the unit ball of $\C^*_u[X]_r $, the comparison maps are 
\begin{align*}
\nu_0\colon  K^{L}_0(C_0(X))&\longrightarrow K^0_{\Omega_X(r,L)}(F), & \nu_1\colon  K^{L}_1(C_0(X))&\longrightarrow K^1_{\Omega_X(r,L)}(F),\\
[p,k]&\longmapsto [\nu(p),\nu(\sigma(p))], & [u]&\longmapsto [\nu(u)],
\end{align*}
where $\Omega_X(r,L)=4rL\cdot \sqrt{N_X(r)}$.
\end{define}

\begin{prop}\label{proe_local_map_well_defined}
    The maps $\nu_*$ are well-defined homomorphisms. They are compatible with the forget-control maps in the following sense: for $0<L\le L'$,
    \[
    \nu_*\circ i^{L,L'}=\phi^{\Omega_X(r,L),\Omega_X(r,L')}\circ\nu_*\colon K^L_*(C_0(X))\longrightarrow K^*_{\Omega_X(r,L')}(F),
    \]
    and for $r\le s$ the composite $K^L_*(C_0(X))\xrightarrow{\nu_*}K^*_{\Omega_X(s,L)}(E_s)\xrightarrow{\phi}K^*_{\Omega_X(s,L)}(E_r)$ agrees with $\phi^{\Omega_X(r,L),\Omega_X(s,L)}\circ\nu_*$.
\end{prop}

\begin{proof}
    We have already seen that $\nu(p)\in M_n(\cK(H)^+)$. We now check the commutator estimates. For any operator $T\in \C_u^*[X]_r$ and any $L$-Lipschitz matrix-valued function $f\in M_n(C_0(X)^+)_L$, the $(x,y)$ component of the commutator matrix is 
    \[ 
    [T\otimes 1_n,f]_{xy}=T_{xy}(f(y)-f(x)),\quad \forall x,y\in X.
    \]
    The propagation bound implies that each row and each column of this matrix has at most $N_X(r)$ nonzero entries. Hence, for each fixed $x$, we have 
    \begin{align*}
         \sum_y\|[T\otimes 1_n,f]_{xy}\|&\le  rL\sum_y|T_{xy}| \\
         &\le rL\sqrt{N_X(r)} \Big(\sum_y|T_{xy}|^2\Big)^{\frac{1}{2}}
         \\ &\le rL\sqrt{N_X(r)}\  \|T\|.
    \end{align*}
    For fixed $y$ we have the same upper bound. We now use the Schur test in the following form: if $A=(A_{xy})_{x,y\in X}$ is a bounded operator on $\ell^2(X)\otimes\C^n$ with entries $A_{xy}\in M_n(\C)$, then
    \begin{equation}\label{eq_schur}
    \|A\|\le\Big(\sup_x\sum_y\|A_{xy}\|\Big)^{1/2}\Big(\sup_y\sum_x\|A_{xy}\|\Big)^{1/2},
    \end{equation}
    which follows from the Cauchy--Schwarz inequality: $$\|A\xi\|^2\le\sum_x\big(\sum_y\|A_{xy}\|\,\|\xi_y\|\big)^2\le\sum_x\big(\sum_y\|A_{xy}\|\big)\big(\sum_y\|A_{xy}\|\,\|\xi_y\|^2\big).$$ Applied to the commutator, \eqref{eq_schur} gives
    \[
    \|[T\otimes 1_n,f]\|\le rL\sqrt{N_X(r)}\ \|T\|.
    \]
    Since the operators are taken from the unit ball, we obtain
    \[ \|[a,p]\|\le rL\sqrt{N_X(r)}<\Omega_X(r,L), \quad \forall a\in E_r,\ p\in M_n(C_0(X)^+)_L. \]
    It remains to check that the comparison maps respect the equivalence relations and block sums. For a $2L$-Lipschitz homotopy the same estimate gives a commutator bound of $2rL\sqrt{N_X(r)}<\Omega_X(r,L)$, so homotopic representatives have homotopic images, and block sums are respected by the formulas on representatives. The two compatibilities with the forget-control maps hold because the formula for $\nu_*$ on representatives involves neither $L$ nor $F$ nor the tolerance.
\end{proof}

\begin{remark}\label{rem_GY_gap}
The proof of \cite[Proposition 4.8]{GuoYu2023Covering} contains a gap for matrix-valued $f$: the inequality $\|f(x)-f(y)\|\ge\max_i\sum_j|f_{ij}(x)-f_{ij}(y)|$ used there fails for the operator norm, and the argument given there only yields a bound that grows linearly with the matrix size $n$. In the discrete case, the Schur test above gives a bound independent of $n$, at the cost of an explicit factor $\sqrt{N_X(r)}$. The main results of \cite{GuoYu2023Covering} are unaffected: the pairing there is applied to Lipschitz $K$-theory classes of a fixed matrix size $n$, and the dependence on $n$ can be absorbed into the pairability condition by taking the Lipschitz constant $L$ smaller, so that only the explicit constants change.
\end{remark}

\subsection{The main theorem and inverse limits}
A metric space $X$ is called $(m,C)$-chain geodesic if 
for every $x,y\in X$ there is a chain of points
\[   
x=x_0,x_1,\dots,x_N=y
\]
satisfying $d(x_i,x_{i+1})\le m$ and $N\le Cd(x,y)$, where $C\ge 1$. 
Examples include word length metrics on finitely generated groups, which are $(1,1)$-chain geodesic. 

Inspired by the pairings in \cite{GuoYu2023Covering} and \cite{Yu_Novikov_Conjecture}, we now state our main theorem, Theorem \ref{thm_main_theorem}, in precise form. It relates the asymptotic behaviour of Lipschitz $K$-theory to that of the quantitative $K$-homology of the uniform Roe algebra. 

\begin{mainthmrestated}
    Let $X$ be a discrete proper metric space with bounded geometry, and suppose that $X$ is $(m,C)$-chain geodesic for some $m>0$ and $C\ge 1$. Let $r\ge \max(m,1)$, $L>0$ satisfy $\Omega_X(r,L)<1/100$, and let $E_r$ be the unit ball of $\C^*_u[X]_r$. Put $M=80Cr\sqrt{N_X(r)}$ and $L'=L/M$. Then the comparison maps 
    $$
    \nu_*\colon K^{L}_*(C_0(X))\longrightarrow K^*_{\Omega_X(r,L)}(E_r)
    $$ 
    are quantitative isomorphisms in the following sense:
    \begin{enumerate}[label=(\arabic*)]
        \item 
        $\forall y\in \im \phi^{\Omega_X(r,L'),\Omega_X(r,L)},\ 
        \exists z\in K^{L}_*(C_0(X)), \text{ such that } \nu_*(z)=y.$
        \item
        $\forall z\in K^{L'}_*(C_0(X)),\ \text{if }\nu_*(z)=0\in K^*_{\Omega_X(r,L')}(E_r),\text{ then } i^{L',L}(z)=0.$
    \end{enumerate}
\end{mainthmrestated}

Passing to inverse limits, the quantitative isomorphism of Theorem \ref{thm_main_theorem} becomes an isomorphism of groups. Recall that $\varprojlim_{L\to 0}K^L_*(C_0(X))$ denotes the group of coherent families $(z_L)_{L>0}$ with $z_L\in K^L_*(C_0(X))$ and $i^{L,L'}(z_L)=z_{L'}$ for $L\le L'$; it is the group of Lipschitz $K$-theory classes that exist, compatibly, at every Lipschitz scale.

\begin{cor}\label{cor_inverse_limit}
Let $X$ be a discrete proper metric space with bounded geometry which is $(m,C)$-chain geodesic, and let $r\ge\max(m,1)$. Then the comparison maps $\nu_*$ induce an isomorphism
\[
\nu_\infty\colon \varprojlim_{L\to 0}K^L_*(C_0(X))\xrightarrow{\ \cong\ }\varprojlim_{\eps\to 0}K^*_\eps(E_r)=QK^*_r(X).
\]
This isomorphism is compatible with the restriction maps $QK^*_s(X)\to QK^*_r(X)$ for $s\ge r\ge\max(m,1)$, which are therefore isomorphisms as well. Consequently,
\[
QK^*(X)\cong QK^*_r(X)\cong\varprojlim_{L\to 0}K^L_*(C_0(X))
\]
for every $r\ge\max(m,1)$.
\end{cor}
\begin{proof}
Write $A_L=K^L_*(C_0(X))$ and $B_\eps=K^*_\eps(E_r)$. The map 
$$L\mapsto\Omega_X(r,L)=4r\sqrt{N_X(r)}\,L$$ is an increasing bijection of $(0,\infty)$, and by Proposition \ref{proe_local_map_well_defined} the maps $\nu_*\colon A_L\to B_{\Omega_X(r,L)}$ commute with the structure maps $i^{L,L''}$ and $\phi^{\Omega_X(r,L),\Omega_X(r,L'')}$ of the two inverse systems. Hence they induce a homomorphism $\nu_\infty$ between the inverse limits. Choose $L_0>0$ with $\Omega_X(r,L_0)<1/100$; since the inverse limits only depend on the systems restricted to $L\le L_0$, it suffices to work with such $L$. Let $M=80Cr\sqrt{N_X(r)}$ be the constant of Theorem \ref{thm_main_theorem}, and note that $\Omega_X(r,L/M)=\Omega_X(r,L)/M$.

\emph{Injectivity.} Let $(z_L)_L$ be a coherent family with $\nu_*(z_L)=0$ for all $L$. For $L\le L_0$, coherence gives $z_L=i^{L/M,L}(z_{L/M})$, and $\nu_*(z_{L/M})=0$ in $B_{\Omega_X(r,L/M)}$, so $z_L=0$ by part (2) of Theorem \ref{thm_main_theorem}. 

\emph{Surjectivity.} Let $(y_\eps)_\eps$ be a coherent family in $\varprojlim_\eps B_\eps$. For $L\le L_0$, the element $y_{\Omega_X(r,L)}=\phi^{\Omega_X(r,L/M),\Omega_X(r,L)}(y_{\Omega_X(r,L/M)})$ lies in the image of the forget-control map, so by part (1) of Theorem \ref{thm_main_theorem} we may choose $z_L\in A_L$ with $\nu_*(z_L)=y_{\Omega_X(r,L)}$. These choices need not be coherent, so put $\tilde z_L:=i^{L/M,L}(z_{L/M})$ for $L\le L_0$. If $L\le L''\le L_0$, then $w:=i^{L/M,L''/M}(z_{L/M})-z_{L''/M}\in A_{L''/M}$ satisfies 
\[
\nu_*(w)=\phi^{\Omega_X(r,L/M),\Omega_X(r,L''/M)}\big(y_{\Omega_X(r,L/M)}\big)-y_{\Omega_X(r,L''/M)}=0,
\]
so $i^{L''/M,L''}(w)=0$ by part (2) of Theorem \ref{thm_main_theorem}, that is, $i^{L,L''}(\tilde z_L)=\tilde z_{L''}$. Thus $(\tilde z_L)_{L\le L_0}$ is a coherent family, and $\nu_*(\tilde z_L)=\phi^{\Omega_X(r,L/M),\Omega_X(r,L)}(y_{\Omega_X(r,L/M)})=y_{\Omega_X(r,L)}$ for all $L\le L_0$. Hence $\nu_\infty(\tilde z)=y$.

\emph{Compatibility in $r$.} For $s\ge r\ge\max(m,1)$ the restriction map $QK^*_s(X)\to QK^*_r(X)$ is induced by the forget-control maps $\phi^{(E_s,\eps)(E_r,\eps)}$, with which $\nu_*$ commutes by Proposition \ref{proe_local_map_well_defined}. Hence the restriction map composed with $\nu_\infty$ for $E_s$ equals $\nu_\infty$ for $E_r$, and since both of these are isomorphisms, so is the restriction map. The last assertion follows because $QK^*(X)$ is the inverse limit of the $QK^*_r(X)$ along these restriction maps.
\end{proof}

\begin{remark}
The argument shows more precisely that $\nu_*$ is an isomorphism of pro-groups between the two inverse systems, after the reindexing $\eps=\Omega_X(r,L)$. In particular it also induces an isomorphism of the derived limits $\varprojlim^1$, which measure the failure of these inverse limits to behave well in exact sequences.
\end{remark}

\section{Proof of the main theorem}
We now prove Theorem \ref{thm_main_theorem}. The main tools are an averaging method and the use of shift operators as test elements. We first illustrate these techniques on a space of two points, and then apply them to general spaces.

Throughout the section, we extend the conjugation action $aTa^*$, for $T\in \cK(H)$, and $a\in U(H)$, to $\cK(H)^+$ by setting $a(T+\lambda I)a^*=aTa^*+\lambda I$.

\subsection{The two-point space}
\label{example_K_one_pt}
Note that when $X=\{\text{pt}\}$, we have $C^*_u(X)=\C$, which commutes with every operator, so $K_{\eps}^*(F)\cong K_*(\C)$ for every $F$. The first interesting case is therefore a space with two points.

\begin{example}\label{exam_2_pts_quant_K}
    Suppose $X=\set{x_1,x_2}$ with $d(x_1,x_2)=1$. Then $C_u^*(X)=M_2(\C)$, and the projections representing classes in quantitative $K$-homology are $p\in M_n(M_2(\C)^+)$. Take the subset 
    \[ S=\Set{ \MAT{1,0;0,1},\MAT{1,0;0,-1},\MAT{-1,0;0,1},\MAT{-1,0;0,-1}}\subseteq \C_u^*[X]_0. 
    \]
    Let $F$ be any subset of the unit ball of $\C_u^*[X]_r$, and let $\eps<1/100$. 
    Then for any $(p,q)\in \cP_{n,\eps}(F\cup S)$, the expectation $p'=\frac{1}{4}\sum_{a\in S} a pa^*$ is exactly the diagonal part of $p$, and it satisfies 
    \[ 
    \Big\|p-\frac{1}{4}\sum_{a\in S} a pa^*\Big\|\le \frac{1}{4}\sum_{a\in S}\|ap-pa\|<\eps,
    \]
    which shows that $p'$ is a $3\eps$-almost projection. By Proposition \ref{prop_almost_proj_unitary}, there is a projection $f(p')\in M_n(M_2(\C)^+)$ with all entries diagonal such that 
    \[  \|p-f(p')\|<7\eps.   \]
    We obtain $f(q')$ in the same way. Since 
    \[
    \|[f(p'),a]\|=\|[p,a]-[p-f(p'),a]\|< \|[p,a]\|+14\eps<20\eps,
    \]
    we see that $[p,q]=[f(p'),f(q')]\in K^0_{20\eps }(F\cup S)$, via the functional calculus applied to a linear homotopy. 
    Therefore, if $r=0$, the image of the forget-control map 
    $$
    \phi\colon  K^0_{\eps}(F\cup S)\to K^0_{20\eps}(F\cup S)\to K^0_{20\eps}(F)
    $$ 
    consists of classes represented by diagonal pairs $(f(p'),f(q'))$. Diagonal pairs are classified, up to homotopy and simultaneous stabilization through diagonal pairs, by the two rank differences of $f(p')$ and $f(q')$ at $x_1$ and at $x_2$, so the image above is the image of a homomorphism from $\Z\oplus \Z$. This homomorphism need not be injective: for $F=\emptyset$ it is the map $\Z\oplus\Z\to K_0(\cK(H))\cong\Z$, $(a,b)\mapsto a+b$. Passing to the inverse limit, however, one finds $QK^0_0(X)\cong\Z\oplus \Z$. Indeed, for $\eps<1/100$ the pointwise rank differences define a homomorphism $\rho_\eps\colon K^0_\eps(E_0)\to\Z\oplus\Z$, $[p,q]\mapsto\big(\mathrm{rank}\,f(p')(x_i)-\mathrm{rank}\,f(q')(x_i)\big)_{i=1,2}$, which is well defined and compatible with the forget-control maps because averaging and rounding are norm continuous and ranks are locally constant. It is surjective, with coherent preimages represented by pairs of projections supported at $x_1$ and at $x_2$, which commute with $E_0$; and its kernel is contained in the kernel of $K^0_\eps(E_0)\to K^0_{20\eps}(E_0)$, because after stabilization two diagonal pairs with equal scalar parts and equal pointwise rank differences are connected by a path of diagonal pairs. Hence $\varprojlim_\eps\rho_\eps\colon QK^0_0(X)\to\Z\oplus\Z$ is an isomorphism.

    The case $r<1$ is the same as $r=0$, since $d(x_1,x_2)=1$. If $r\ge 1$, then there is a shift 
    $$ V=\MAT{0,1;0,0}\in \C_u^*[X]_1. $$
    Let $F$ be a subset of $E_1$. The forget-control map 
    \[
    \phi\colon  K^0_{\eps}(F\cup S\cup \{V\})\to K^0_{20\eps}(F\cup S\cup \{V\})  
    \]
    still sends $[p,q]$ to $[f(p'),f(q')]$, since the path between $(p,q)$ and $(f(p'),f(q'))$ is constructed from norm closeness, and hence also almost commutes with $V$. The additional condition that $f(p')$ almost commutes with $V$ at level $20\eps$ implies that $\|f(p')(x_1)-f(p')(x_2)\|<20\eps<1$, so $f(p')$ has the same rank at $x_1$ and at $x_2$. Moreover, the linear path from $f(p')$ to the diagonal projection taking the value $f(p')(x_1)$ at both points, which commutes with all of $M_2(\C)$, consists of $100\eps^2$-almost projections almost commuting with $F$ at level $20\eps$; applying Proposition \ref{prop_almost_proj_unitary} along it as before gives a path at level $24\eps$. In particular the image of 
    \[ 
    \phi\colon  K^0_{20\eps}(F\cup S\cup \{V\})\to K^0_{200\eps}(F)  
    \]
    (the constant $200$ is a generous upper bound) consists of classes represented by pairs of projections commuting with $M_2(\C)$. Consequently, on the image of $K^0_\eps(E_1)$ the homomorphism $\rho_\eps$ takes values in the diagonal copy of $\Z$, and the same argument as in the case $r=0$ shows that $QK_1^0(X)\cong \Z$.
\end{example}

\subsection{Diagonal signs and the expectation}
\label{section_general_proof}
We now prove Theorem \ref{thm main}. Let $X$ be a discrete proper metric space with bounded geometry which is $(m,C)$-chain geodesic. Let $r\ge \max(m,1)$ and $L>0$ satisfy $\Omega_X(r,L)<1/100$, and let $E_r$ be the unit ball of $\C_u^*[X]_r$. The proof rests on two families of test operators in $E_r$: diagonal sign operators, which force an almost commuting projection or unitary to be close to a diagonal one, and shifts of length at most $m$, which control the variation of a diagonal operator between nearby points. We treat these two steps in separate lemmas.

Let $\chi_x$ be the projection in $\cB(H)$ onto $\delta_x$. Consider the global sign group
\[
S_X=\Set{ \sum_{x\in X}\theta_x\chi_x\in \cB(H)\mid  \theta_x\in \{1,-1\} \text{ for all } x\in X},
\]
where the sum converges in the strong operator topology. This is a compact group in the strong operator topology, and hence carries a normalized Haar measure $\mu_S$. Every element of $S_X$ is a self-adjoint unitary of propagation $0$, so $S_X\subseteq E_r$ for every $r\geq 0$. Note that the scalar part of $T\in M_n(\cK(H)^+)$ commutes with $S_X$, and that compact operators can be approximated by finite rank operators, so the map $s\mapsto sTs^*$ is continuous in operator norm. Hence we may define the expectation
\[
\cE(T)=\int_{S_X} sTs^*\, d\mu_S \ \in M_n(\cK(H)^+)
\]
as a Bochner integral. Since $ss^*=1$ for every $s\in S_X$ and $\mu_S$ is a probability measure, we have $\int_{S_X} Tss^*\, d\mu_S=T$, and therefore
\begin{equation}\label{eq_expectation_estimate}
\|\cE(T)-T\|\le\sup_{s\in S_X} \|[T,s]\|.
\end{equation}
Averaging over the signs kills all off-diagonal matrix entries, so $\cE(T)$ is the block diagonal part of $T$. In particular $\cE(T)-\sigma(T)\in M_n\big(\bigoplus_{x\in X }\C\big)$, and since the diagonal entries of a compact operator vanish at infinity, we may regard $\cE(T)$ as an element of $M_n(C_0(X)^+)$, that is, as a matrix-valued function on $X$. We write $\cE(T)(x)\in M_n(\C)$ for its value at $x$. The map $\cE$ is linear, contractive, $*$-preserving and norm continuous, and $\cE(T)=T$ precisely when $T$ is diagonal.

Next, for $x\neq y$ in $X$ let $V_{xy}\in\cB(H)$ be the partial isometry
\[
V_{xy}h= \<h,\delta_x\rat\delta_y,\qquad h\in H,
\]
sending $\delta_x$ to $\delta_y$. It has norm one and propagation $d(x,y)$, so $V_{xy}\in E_r$ whenever $d(x,y)\le r$. If $T=(T_{xy})_{x,y\in X}\in M_n(\cK(H)^+)$ with $T_{xy}\in M_n(\C)$, a direct computation shows that the $(y,x)$ entry of $[T,V_{xy}]$ is $T_{yy}-T_{xx}$, and hence
\begin{equation}\label{eq_shift_estimate}
\|\cE(T)(y)-\cE(T)(x)\|\le \|[T,V_{xy}]\|.
\end{equation}
If $T$ is itself diagonal, then $[T,V_{xy}]$ has only this one nonzero entry, and \eqref{eq_shift_estimate} is an equality.

\begin{lemma}[Controlled diagonalization]\label{lem_controlled_diag}
Let $0<\alpha<1/100$, and let $F\subseteq E_r$ be a subset containing $S_X$.
\begin{propparts}
\item If $(p,q)\in\cP_{n,\alpha}(F)$, then there is a pair $(p^\flat,q^\flat)\in\cP_{n,20\alpha}(F)$ of projections in $M_n(C_0(X)^+)$ with $\sigma(p^\flat)=\sigma(p)$, $\sigma(q^\flat)=\sigma(q)$, $\|p^\flat-p\|<5\alpha$ and $\|q^\flat-q\|<5\alpha$, which is connected to $(p,q)$ by a path in $\cP_{n,20\alpha}(F)$ along which the scalar parts are constant. If $p$ (resp.\ $q$) is already diagonal, then $p^\flat=p$ (resp.\ $q^\flat=q$).
\item If $u\in\cU_{n,\alpha}(F)$, then there is a unitary $u^\flat\in\cU_{n,20\alpha}(F)\cap M_n(C_0(X)^+)$ with $\sigma(u^\flat)=\sigma(u)$ and $\|u^\flat-u\|<5\alpha$, which is connected to $u$ by a path in $\cU_{n,20\alpha}(F)$. If $u$ is already diagonal, then $u^\flat=u$.
\item The constructions in (1) and (2) are norm continuous. Consequently, if $(p_t,q_t)_{t\in[0,1]}$ is a path in $\cP_{n,\alpha}(F)$, then $(p_t^\flat,q_t^\flat)$ is a path in $\cP_{n,20\alpha}(F)$ consisting of diagonal pairs, and similarly for paths of unitaries.
\end{propparts}
\end{lemma}
\begin{proof}
For (1), put $p'=\cE(p)$. By \eqref{eq_expectation_estimate} and $S_X\subseteq F$ we have $\|p'-p\|\le\sup_{s\in S_X}\|[p,s]\|<\alpha$, where the last inequality is strict because $S_X$ is compact and $s\mapsto\|[p,s]\|$ is norm continuous, so that the supremum is attained; moreover $p'$ is self-adjoint with $\sigma(p')=\sigma(p)$. Writing 
$$p'^2-p'=(p'-p)p'+(p-I)(p'-p),$$
we get $\|p'^2-p'\|<2\alpha<1/4$. Let $f$ be as in Proposition \ref{prop_almost_proj_unitary}(1) for $\varepsilon=2\alpha$, and put $p^\flat:=f(p')$. Then $p^\flat$ is a projection with $\|p^\flat-p'\|<4\alpha$, hence $\|p^\flat-p\|<5\alpha$. Since $p'$ is diagonal, so is $p^\flat$, that is, $p^\flat\in M_n(C_0(X)^+)$, and $\sigma(p^\flat)=f(\sigma(p))=\sigma(p)$ because $\sigma(p)$ is a projection. For $a\in F$ we have
\[
\|[p^\flat,a]\|\le\|[p,a]\|+2\|p^\flat-p\|<11\alpha.
\]
Now consider the linear path $v_t=(1-t)p+tp^\flat$, $t\in[0,1]$. Since $p$ and $p^\flat$ are projections, $v_t^2-v_t=-t(1-t)(p-p^\flat)^2$, so that $\|v_t^2-v_t\|\le\frac14\|p-p^\flat\|^2<7\alpha^2$, while $\|[v_t,a]\|<11\alpha$ by convexity and $\sigma(v_t)=\sigma(p)$. By Proposition \ref{prop_almost_proj_unitary}(1), $t\mapsto f(v_t)$ is a norm continuous path of projections with $\|f(v_t)-v_t\|<14\alpha^2$, $f(v_0)=p$, $f(v_1)=p^\flat$, $\sigma(f(v_t))=\sigma(p)$, and
\[
\|[f(v_t),a]\|\le\|[v_t,a]\|+2\|f(v_t)-v_t\|<11\alpha+28\alpha^2<12\alpha
\]
for all $a\in F$, using $\alpha<1/100$. Doing the same for $q$ gives the path required in (1). If $p$ is diagonal, then $p'=p$ and $p^\flat=f(p)=p$.

For (2), put $v=\cE(u)$, so that $\|v-u\|<\alpha$ as in (1), and $\sigma(v)=\sigma(u)$. Then $v=u\big(1+u^*(v-u)\big)$ is invertible, and
\[
\|1-v^*v\|=\|u^*u-v^*v\|\le\|u^*(u-v)\|+\|(u-v)^*v\|<2\alpha<1/4,
\]
so by Proposition \ref{prop_almost_proj_unitary}(2) the polar part $u^\flat:=v(v^*v)^{-1/2}$ is a unitary with $\|u^\flat-v\|<4\alpha$, hence $\|u^\flat-u\|<5\alpha$. Since $v$ is diagonal, so is $u^\flat$, and $\sigma(u^\flat)$ is the polar part of $\sigma(v)=\sigma(u)$, which is $\sigma(u)$ itself since $\sigma(u)$ is unitary. For $a\in F$ we get $\|[u^\flat,a]\|<11\alpha$ and $\|[(u^\flat)^*,a]\|<11\alpha$ as in (1). Consider the linear path $w_t=(1-t)u+tu^\flat$. As $\|w_t-u\|\le\|u^\flat-u\|<1$, each $w_t$ is invertible, and the computation in the proof of Corollary \ref{cor_norm_closed_proj_uni} gives $1-w_t^*w_t=t(1-t)(u-u^\flat)^*(u-u^\flat)$, so that $\|1-w_t^*w_t\|<7\alpha^2$. By Proposition \ref{prop_almost_proj_unitary}(2), the polar parts $w_t(w_t^*w_t)^{-1/2}$ form a norm continuous path of unitaries from $u$ to $u^\flat$ within $14\alpha^2$ of $w_t$, and hence, as in (1), almost commuting with every $a\in F$ at level $11\alpha+28\alpha^2<12\alpha$, together with their adjoints. If $u$ is diagonal, then $v=u$ and $u^\flat=u$.

For (3), note that $\cE$ is norm continuous, that the continuous functional calculus $T\mapsto f(T)$ is norm continuous on self-adjoint elements, and that the polar part is norm continuous on invertible elements.
\end{proof}

\subsection{Lipschitz control from shifts}
\begin{lemma}[Lipschitz control from shifts]\label{lem_shift_lipschitz}
Suppose that $X$ is $(m,C)$-chain geodesic, let $\alpha>0$, and let $g\in M_n(C_0(X)^+)$ satisfy $\|[g,V_{xy}]\|<\alpha$ whenever $0<d(x,y)\le m$. Then $g$ is $C\alpha$-Lipschitz.
\end{lemma}
\begin{proof}
Since $g$ is diagonal, \eqref{eq_shift_estimate} gives $\|g(x)-g(y)\|<\alpha$ whenever $0<d(x,y)\le m$. Given $x\neq y$ in $X$, choose a sequence $x=x_0,x_1,\dots,x_N=y$ with $d(x_i,x_{i+1})\le m$ and $N\le Cd(x,y)$. Then 
\[
\|g(x)-g(y)\|\le\sum_{i=0}^{N-1}\|g(x_i)-g(x_{i+1})\|\le N\alpha\le C\alpha\, d(x,y). \qedhere
\]

\end{proof}

\subsection{Proof of Theorem \ref{thm_main_theorem}}
\begin{proof}[Proof of Theorem \ref{thm_main_theorem}]
Put $\cV_m=\Set{V_{xy}\mid 0<d(x,y)\le m}$. Since $m\le r$, we have $S_X\cup\cV_m\subseteq E_r$. With $M=80Cr\sqrt{N_X(r)}$ and $L'=L/M$ as in the statement, set
\[
\eps:=\Omega_X(r,L')=4rL'\sqrt{N_X(r)}=\frac{L}{20C},
\]
so that $L=20C\eps$ and $\Omega_X(r,L)=M\eps$. Note that $M\ge 80$, so $L'\le L$ and the forget-control map $i^{L',L}$ exists, $20\eps\le \Omega_X(r,L)$, and $\eps<1/(100M)<1/8000$.

\emph{Even degree, surjectivity.} Let $y$ be in the image of $\phi\colon K^0_{\eps}(E_r)\to K^0_{\Omega_X(r,L)}(E_r)$. By Lemma \ref{lem_element_to_[p,e_n]}, $y$ has a preimage of the form $[p,1_n\oplus 0_n]$ with $(p,1_n\oplus 0_n)\in \cP_{2n,\eps}(E_r)$ and $\sigma(p)=1_n\oplus 0_n$. By Lemma \ref{lem_controlled_diag}(1), $(p,1_n\oplus 0_n)$ is connected in $\cP_{2n,20\eps}(E_r)$ to $(p^\flat,1_n\oplus 0_n)$, where $p^\flat\in M_{2n}(C_0(X)^+)$ is a projection with $\sigma(p^\flat)=1_n\oplus 0_n$. Since $p^\flat$ almost commutes with $\cV_m$ at level $20\eps$, Lemma \ref{lem_shift_lipschitz} shows that $p^\flat$ is $20C\eps=L$-Lipschitz, so $z:=[p^\flat,n]\in K^L_0(C_0(X))$. As $20\eps\le \Omega_X(r,L)$, the classes $[p,1_n\oplus 0_n]$ and $[p^\flat,1_n\oplus 0_n]=\nu_0(z)$ have the same image in $K^0_{\Omega_X(r,L)}(E_r)$, that is, $\nu_0(z)=y$.

\emph{Even degree, injectivity.} Let $z\in K^{L'}_0(C_0(X))$ satisfy $\nu_0(z)=0$ in $K_{\eps}^0(E_r)$. By Lemma \ref{lem_LipK_prescribe}, $z=[p,n]$ with $p\in P^{L'}_{2n}(X)$ and $\sigma(p)=1_n\oplus 0_n$, so that $\nu_0(z)=[p,1_n\oplus 0_n]$ with $(p,1_n\oplus 0_n)\in\cP_{2n,\eps}(E_r)$. By Lemma \ref{lem_K0_element_to_be_zero}, there are $k\in\N$ and a path $(P_t,Q_t)_{t\in[0,1]}$ in $\cP_{2n+2k,\eps}(E_r)$ from $(p\oplus 1_k\oplus 0_k,\ 1_n\oplus 0_n\oplus 1_k\oplus 0_k)$ to a pair of the form $(s,s)$. Applying Lemma \ref{lem_controlled_diag} along this path yields a path $(P_t^\flat,Q_t^\flat)$ of diagonal projection pairs in $\cP_{2n+2k,20\eps}(E_r)$, with the same initial point because $p\oplus 1_k\oplus 0_k$ and $1_n\oplus 0_n\oplus 1_k\oplus 0_k$ are already diagonal, and with $P_1^\flat=s^\flat=Q_1^\flat$. Following $P_t^\flat$ and then $Q_t^\flat$ in reverse gives a norm continuous path of projections in $M_{2n+2k}(C_0(X)^+)$ from $p\oplus 1_k\oplus 0_k$ to $1_n\oplus 0_n\oplus 1_k\oplus 0_k$, each of which almost commutes with $\cV_m$ at level $20\eps$ and is therefore $L$-Lipschitz by Lemma \ref{lem_shift_lipschitz}. Hence $p\oplus 1_k\oplus 0_k$ and $1_n\oplus 0_n\oplus 1_k\oplus 0_k$ are homotopic in $P^{L}_{2n+2k}(X)\subseteq P^{2L}(X)$, so that $[p\oplus 1_k\oplus 0_k,n+k]=[1_n\oplus 0_n\oplus 1_k\oplus 0_k,n+k]$ in $K^L_0(C_0(X))$. Both $[1_k\oplus 0_k,k]$ and $[1_n\oplus 0_n\oplus 1_k\oplus 0_k,n+k]$ are zero in $K^L_0(C_0(X))$, being classes of scalar projections, and therefore
\[
i^{L',L}(z)=[p,n]=[p\oplus 1_k\oplus 0_k,n+k]-[1_k\oplus 0_k,k]=0.
\]

\emph{Odd degree.} For surjectivity, let $y$ be in the image of $K^1_{\eps}(E_r)\to K^1_{\Omega_X(r,L)}(E_r)$, represented by some $u\in\cU_{n,\eps}(E_r)$, which by Lemma \ref{lem_precribe_scalar_P1_U1}(3) we may take with $\sigma(u)=1_n$. By Lemma \ref{lem_controlled_diag}(2), $u$ is connected in $\cU_{n,20\eps}(E_r)$ to a diagonal unitary $u^\flat$ with $\sigma(u^\flat)=1_n$, and $u^\flat$ is $L$-Lipschitz by Lemma \ref{lem_shift_lipschitz}. Hence $u^\flat\in U^L_n(X)$. Since $20\eps\le\Omega_X(r,L)$ we have $u^\flat\in\cU_{n,\Omega_X(r,L)}(E_r)$, and the connecting path lies in $\cU_{n,2\Omega_X(r,L)}(E_r)$, so $\nu_1([u^\flat])=y$.

For injectivity, let $z=[u]\in K^{L'}_1(C_0(X))$ with $u\in U^{L'}_n(X)$, and suppose that $\nu_1(z)=0$ in $K^1_{\eps}(E_r)$. By definition of $K^1_\eps(E_r)$, after stabilizing there is a path $(u_t)_{t\in[0,1]}$ in $\cU_{n+k,2\eps}(E_r)$ from $u\oplus 1_k$ to $1_{n+k}$. Applying Lemma \ref{lem_controlled_diag}(2) with $\alpha=2\eps<1/100$ along this path yields a path $(u_t^\flat)$ of diagonal unitaries in $\cU_{n+k,40\eps}(E_r)$ from $u\oplus 1_k$ to $1_{n+k}$, both endpoints being diagonal already. By Lemma \ref{lem_shift_lipschitz}, each $u_t^\flat$ is $40C\eps=2L$-Lipschitz. Thus $u\oplus 1_k$ is homotopic to $1_{n+k}$ in $U^{2L}(X)$, which means precisely that $i^{L',L}(z)=[u]=0$ in $K^L_1(C_0(X))$. Note that the factor $2$ in the definition of the odd groups is what allows the same constant $M$ to work in both degrees.
\end{proof}

\section{Pairing with quantitative \textit{K}-theory}
In this section, we establish a pairing in degree $0$ between quantitative $K$-homology and the quantitative $K$-theory of $C_u^*(X)$, similar to the difference construction in \cite{GuoYu2023Covering}. Recall that the quantitative $K$-theory of the uniform Roe algebra is defined using quasi-idempotents. 

\subsection{Quantitative \textit{K}-theory}
\begin{define}
\label{def_quantitative_K}
Let $0<\varepsilon<\frac{1}{20}$, $r>0$, and $N\geq 1$. An element $e=(e_{ij})\in M_n(C^*_u(X))$ is called an \emph{$(\varepsilon,r,N)$-quasi-idempotent} if
$$
\|{e^2-e}\|<\varepsilon,\qquad e_{ij}\in \C^*_u[X]_r,\qquad\max(\|{e}\|,\|{1-e}\|)\leq N.
$$
Let $\textnormal{Idem}_n^{\varepsilon,r,N}(C^*_u(X))$ be the subspace of $(\varepsilon,r,N)$-quasi-idempotents in $M_n(C^*_u(X))$. For each positive integer $n$ there is an injection 
$$\textnormal{Idem}^{\varepsilon,r,N}_n(C^*_u(X))\hookrightarrow\textnormal{Idem}^{\varepsilon,r,N}_{n+1}(C^*_u(X))
$$ 
given by inclusion into the upper-left corner, with respect to which we define 
$$
\textnormal{Idem}_\infty^{\varepsilon,r,N}(C^*_u(X))=\varinjlim \textnormal{Idem}^{\varepsilon,r,N}_n(C^*_u(X)),
$$
where the direct limit is taken in the category of topological spaces.
Define an equivalence relation $\sim$ on $\textnormal{Idem}_\infty^{\varepsilon,r,N}(C^*_u(X))$ by declaring $e\sim f$ if $e$ and $f$ are homotopic through $(4\varepsilon,r,4N)$-quasi-idempotents. Denote the equivalence class of an element $e\in\textnormal{Idem}_\infty^{\varepsilon,r,N}(C^*_u(X))$ by $[e]$, and define
$$
[e]+[f]=\left[\begin{pmatrix}e&0\\0&f\end{pmatrix}\right].
$$
Under this operation, $\textnormal{Idem}_\infty^{\varepsilon,r,N}(C^*_u(X))/\sim$ becomes an abelian monoid with identity $[0]$. Let $K_0^{\varepsilon,r,N}(C^*_u(X))$ be its Grothendieck completion. 
\end{define}

\subsection{The pairing}
Throughout this section, $X$ is a discrete proper metric space with bounded geometry, $r>0$, and $E_r$ is the unit ball of $\C_u^*[X]_r$. The chain geodesic hypothesis is not needed for the pairing. The key input is the following commutator estimate, which is independent of the sizes of the matrices involved; this uniformity is what the full test set $E_r$ provides.

\begin{lemma}\label{lem_pairing_commutator}
Let $\eps>0$ and $N\ge 1$. Let $p\in M_n(\cK(H)^+)$ be a projection with $\|[p,a]\|<\eps$ for all $a\in S_X\cup \set{V_{xy}\mid 0<d(x,y)\le r}$ (for instance, for all $a\in E_r$), and let $e\in M_k(\C_u^*[X]_r)$ with $\|e\|\le N$. Then, with the natural amplifications,
\[
\|[p\otimes I_k,\ I_n\otimes e]\|\le \big(N_X(r)+2\big)N\eps\le 3N_X(r)N\eps.
\]
\end{lemma}
\begin{proof}
Let $\cE$ be the expectation of Section \ref{section_general_proof} and put $p'=\cE(p)\in M_n(C_0(X)^+)$. By \eqref{eq_expectation_estimate} we have $\|p-p'\|\le\eps$, and by \eqref{eq_shift_estimate} we have $\|p'(x)-p'(y)\|<\eps$ whenever $0<d(x,y)\le r$. Write $e=(e_{xy})_{x,y\in X}$ with $e_{xy}\in M_k(\C)$, so that $e_{xy}=0$ when $d(x,y)>r$ and $\|e_{xy}\|\le N$. The $(x,y)$ entry of $[p'\otimes I_k, I_n\otimes e]$ is $(p'(x)-p'(y))\otimes e_{xy}$, which has norm at most $N\eps$, and each row and each column contains at most $N_X(r)$ nonzero entries. The Schur test \eqref{eq_schur} therefore gives $\|[p'\otimes I_k,I_n\otimes e]\|\le N_X(r)N\eps$. Finally $\|[(p-p')\otimes I_k, I_n\otimes e]\|\le 2\|p-p'\|\,\|e\|\le 2N\eps$, and the triangle inequality gives the claim.
\end{proof}

An element $a$ of a $C^*$-algebra $A$ is a generalized idempotent if $\spec(a)$ is contained in the union of $\{\Re z>1/2\}$ and $\{\Re z<1/2\}$. For such $a$, the holomorphic functional calculus $h(a)$ with respect to the function $h$ below produces an actual idempotent; see \cite[Lemma 11.2.9]{rordam2000introduction}. 
\begin{equation*}
    h(z)=
    \begin{cases}
    0, & z\in \{\Re z<1/2\},
    \\ 1, & z\in \{\Re z>1/2\}.
    \end{cases}
\end{equation*}
We shall use the following criterion: if $\|a^2-a\|<1/4$, then $a$ is a generalized idempotent. Indeed, if $\lambda=\frac12+it\in\spec(a)$, then $\lambda^2-\lambda=-(\frac14+t^2)\in\spec(a^2-a)$ would have modulus at least $1/4$. Moreover, $h$ is norm continuous on the set of generalized idempotents, and preserves direct sums.

We now define the pairing on representatives. Let $p,q\in M_l(\cK(H)^+)$ be projections with $\sigma(p)=\sigma(q)$, so that $p-q\in M_l(\cK(H))$, and let $e\in M_k(\C_u^*[X]_r)$. Define
\[
J_{p,q}(e)=((p-q)\otimes I_k)(I_{l}\otimes e)+q\otimes I_k \ \in M_{lk}(\cK(H)^+),
\]
which satisfies $\sigma(J_{p,q}(e))=\sigma(q)\otimes I_k$. For a normalized pair $(p,1_n\oplus 0_n)$ as in Lemma \ref{lem_element_to_[p,e_n]} we write $q_n=1_n\oplus 0_n$ and $I_p(e)=J_{p,q_n}(e)$. Suppressing the amplifications, a direct computation using $p^2=p$ and $q^2=q$ gives
\[
J_{p,q}(e)^2-J_{p,q}(e)=(p-q)^2(e^2-e)-(p-q)[p,e]e-(p-q)[q,e](1-e).
\]
Since $\|p-q\|\le 1$ for projections $p,q$, Lemma \ref{lem_pairing_commutator} yields the following estimate: if $(p,q)\in\cP_{l,\eps}(E_r)$ with $\sigma(p)=\sigma(q)$, and $e$ satisfies $\|e^2-e\|<\delta$ and $\max(\|e\|,\|1-e\|)\le N$, then
\begin{equation}\label{eq_J_estimate}
\|J_{p,q}(e)^2-J_{p,q}(e)\|< \delta+6N_X(r)N^2\eps.
\end{equation}
For an $(\eps,r,N)$-quasi-idempotent $e$ we take $\delta=\eps$, and for the $(4\eps,r,4N)$-quasi-idempotents occurring in the homotopies of Definition \ref{def_quantitative_K} we take $\delta=4\eps$ and replace $N$ by $4N$. In both cases, as $N_X(r),N\ge 1$,
\[
\|J_{p,q}(e)^2-J_{p,q}(e)\|< 4\eps+96N_X(r)N^2\eps\le 100N_X(r)N^2\eps.
\]
Hence, under the standing assumption
\begin{equation}\label{eq_pairing_hypothesis}
N^2\eps<\frac{1}{10^5N_X(r)},
\end{equation}
every such $J_{p,q}(e)$ is a generalized idempotent, $h(J_{p,q}(e))$ is an idempotent in $M_{lk}(\cK(H)^+)$ with scalar part $\sigma(q)\otimes I_k$, and we may define
\[
D_{p,q}(e)=[h(J_{p,q}(e))]-[q\otimes I_k] \ \in K_0(\cK(H))\cong\Z,
\]
where the difference lies in $K_0(\cK(H))=\ker\big(K_0(\cK(H)^+)\to K_0(\C)\big)$ because both idempotents have the same scalar part. We write $D_p(e)=D_{p,q_n}(e)$ for normalized pairs.

\begin{define}\label{def_pairing}
Under the assumption \eqref{eq_pairing_hypothesis}, define a pairing between $K^0_\eps(E_r)$ and $K_0^{\eps,r,N}(C_u^*(X))$ by  
\[ 
\langle[p,1_n\oplus 0_n],[e]-[f]\rangle=D_p(e)-D_p(f)\in K_0(\cK(H))\cong\Z. 
\]
\end{define}

\begin{prop}\label{prop_pairing_well_defined}
     Under the assumption \eqref{eq_pairing_hypothesis}, the above pairing is well defined and bi-additive. 
\end{prop}
\begin{proof}
    We verify in turn that $D_{p,q}(e)$ respects each of the equivalences involved.

    \emph{Homotopies and direct sums of quasi-idempotents.} Fix $(p,q)\in\cP_{l,\eps}(E_r)$ with $\sigma(p)=\sigma(q)$. If $(e_t)_{t\in[0,1]}$ is a path of $(4\eps,r,4N)$-quasi-idempotents, then by \eqref{eq_J_estimate} $t\mapsto J_{p,q}(e_t)$ is a norm continuous path of generalized idempotents, so $t\mapsto h(J_{p,q}(e_t))$ is a norm continuous path of idempotents in $M_{lk}(\cK(H)^+)$ and $D_{p,q}(e_0)=D_{p,q}(e_1)$. Next, $J_{p,q}(e_1\oplus e_2)$ is conjugate to $J_{p,q}(e_1)\oplus J_{p,q}(e_2)$ by a permutation of coordinates, which is a scalar unitary, and $h$ preserves direct sums; hence $D_{p,q}(e_1\oplus e_2)=D_{p,q}(e_1)+D_{p,q}(e_2)$. Also $J_{p,q}(0)=q\otimes I_k$ is a projection, so $h(J_{p,q}(0))=q\otimes I_k$ and $D_{p,q}(0)=0$; in particular $D_{p,q}(e\oplus 0)=D_{p,q}(e)$. Thus $D_{p,q}$ is a well-defined additive map on the monoid $\textnormal{Idem}_\infty^{\varepsilon,r,N}(C^*_u(X))/\sim$, and by the universal property of the Grothendieck completion it extends to a homomorphism $K_0^{\eps,r,N}(C_u^*(X))\to K_0(\cK(H))$, given by $[e]-[f]\mapsto D_{p,q}(e)-D_{p,q}(f)$.

    \emph{Homotopies of projection pairs.} Fix an $(\eps,r,N)$-quasi-idempotent $e$, or more generally a $(4\eps,r,4N)$-quasi-idempotent. If $(p_t,q_t)_{t\in[0,1]}$ is a path in $\cP_{l,\eps}(E_r)$ with $\sigma(p_t)=\sigma(q_t)$ for all $t$, then by \eqref{eq_J_estimate} $t\mapsto h(J_{p_t,q_t}(e))$ is a norm continuous path of idempotents, while $t\mapsto q_t\otimes I_k$ is a norm continuous path of projections. Hence $D_{p_0,q_0}(e)=D_{p_1,q_1}(e)$.

    \emph{Stabilization and direct sums of pairs.} If $s\in M_{l'}(\cK(H)^+)$ is a projection, then, up to conjugation by a permutation of coordinates,
    \[
    J_{p\oplus s,\,q\oplus s}(e)=J_{p,q}(e)\oplus J_{s,s}(e)=J_{p,q}(e)\oplus (s\otimes I_k),
    \]
    and since $h(s\otimes I_k)=s\otimes I_k$ we get $D_{p\oplus s,\,q\oplus s}(e)=D_{p,q}(e)$. More generally, $D_{p_1\oplus p_2,\,q_1\oplus q_2}(e)=D_{p_1,q_1}(e)+D_{p_2,q_2}(e)$.

    \emph{Conclusion.} Let $(p,q_n)$ and $(\wt p,q_m)$ be two normalized representatives of the same class in $K^0_\eps(E_r)$. Unwinding the equivalence relation of Definition \ref{def_K_hom_K0}, there are projections $s,\wt s$ such that $(p\oplus s,\,q_n\oplus s)$ and $(\wt p\oplus\wt s,\,q_m\oplus \wt s)$ are connected by a path in $\cP_{\infty,\eps}(E_r)$: indeed, the relation ``$x$ and $y$ become path connected after stabilizing each of them'' contains the two generating relations and is transitive, since stabilizations can be combined and direct summands can be reordered by scalar rotations, so it coincides with the generated relation. Note that $D_{p,q}(e)$ is invariant under simultaneous conjugation of $(p,q)$ by a scalar unitary $w$, because $J_{wpw^*,wqw^*}(e)=(w\otimes I_k)J_{p,q}(e)(w\otimes I_k)^*$. We may therefore conjugate the endpoints of the path by scalar unitaries, and then by Lemma \ref{lem_precribe_scalar_P1_U1}(2) the path may be taken to consist of pairs with equal scalar parts. By the previous two steps, $D_p(e)=D_{p\oplus s,\,q_n\oplus s}(e)=D_{\wt p\oplus \wt s,\,q_m\oplus\wt s}(e)=D_{\wt p}(e)$. Together with the first step this shows that the pairing is well defined. It is additive in the second variable by the first step, and additive in the first variable since $[p_1,q_{n_1}]+[p_2,q_{n_2}]=[p_1\oplus p_2,q_{n_1}\oplus q_{n_2}]$, a rotation by scalar unitaries brings $q_{n_1}\oplus q_{n_2}$ to the normalized form $q_{n_1+n_2}$ through pairs with equal scalar parts, and $D_{p_1\oplus p_2,\,q_{n_1}\oplus q_{n_2}}(e)=D_{p_1}(e)+D_{p_2}(e)$ by the third step.
\end{proof}

\subsection{Compatibility with the pairing of Guo and Yu}
In \cite[Section 4]{GuoYu2023Covering}, a pairing
\[
\<\cdot,\cdot\rat_{\eps,r,N,L}\colon K_0^{\eps,r,N}(C^*(M))\times K_0^L(C_0(M))\longrightarrow K_0(\cK)
\]
    between quantitative $K$-theory of the Roe algebra and Lipschitz $K$-theory is constructed for a Riemannian manifold $M$; it is defined when the quadruple $(\eps,r,N,L)$ is \emph{pairable}, meaning that $2^6rN^2L+\eps<2^{-68}N^{-16}$. As remarked there, the construction makes sense for metric spaces, and it applies verbatim to the uniform Roe algebra $C_u^*(X)$ acting on $H=\ell^2(X)$; in this setting $C_0(X)\subseteq\cK(H)$, so the intermediate ideal $\langle\cK\cup C_0(M)\rangle$ used there is simply $\cK(H)$. That pairing is defined by the difference construction of \cite{KasparovYu2006Coarse}: for $\alpha=[e]-[f]$ and $z=[p,n]$ normalized as in Lemma \ref{lem_LipK_prescribe}, one applies the difference construction $d(\cdot,\cdot)$ first to the pair $\big((I\otimes e)(p\otimes I),(I\otimes e)(q_n\otimes I)\big)$, and likewise for $f$, and then to the two resulting quasi-idempotents, before passing to $K_0(\cK)$ by the holomorphic functional calculus.

\begin{remark}\label{rem_GY_compatibility}
We expect that the above pairing is compatible with the pairing of Guo and Yu, in the following sense.
Let $(\eps,r,N,L)$ be pairable with $\eps\le\Omega_X(r,L)$, and let $L$ be sufficiently small, depending on $r$, $N$ and $N_X(r)$. Then for all $z\in K^L_0(C_0(X))$ and $\alpha\in K_0^{\eps,r,N}(C_u^*(X))$ one should have
\[
\<\nu_0(z),\alpha\rat=\<\alpha,z\rat_{\eps,r,N,L}\ \in K_0(\cK(H))\cong\Z,
\]
where on the left-hand side $\alpha$ is regarded as an element of $K_0^{\Omega_X(r,L),r,N}(C_u^*(X))$ via the forget-control map. We indicate the argument, but do not carry out the verification in detail. Write $z=[p,n]$ with $\sigma(p)=q_n$ and $\alpha=[e]-[f]$; both pairings are additive in $\alpha$, so it suffices to compare the contributions of $e$. Suppressing amplifications, both sides are classes in $K_0(\cK(H))$ of idempotents obtained by holomorphic functional calculus from explicit almost idempotents, and the claim reduces to the following three facts, each of which requires a homotopy through quasi-idempotents whose defect is controlled by $\|e^2-e\|$, $\|[p,e]\|$ and $N$, hence by Lemma \ref{lem_pairing_commutator} is small when $\eps$ and $L$ are small; the bookkeeping of constants, which determines the precise smallness condition, is what remains to be checked. First, the difference construction $d(a_1,a_2)$ for quasi-idempotents $a_1,a_2$ with $a_1-a_2\in\cK(H)$ represents the relative class of the pair $(a_1,a_2)$ in $K_0(\cK(H))$, so that $\<[e],z\rat_{\eps,r,N,L}$ is the relative class of $(ep,eq_n)$. Second, $ep$ and $pe$ are equivalent almost idempotents, with the equivalence implemented by the pair $(pe,ep)$ itself, since $(pe)(ep)$ and $(ep)(pe)$ are close to $pe$ and $ep$ respectively; hence the relative class of $(ep,eq_n)$ equals that of $(pe,q_ne)$. Third, $I_p(e)=pe+q_n(1-e)$ and $q_n=q_ne+q_n(1-e)$ are sums of almost orthogonal almost idempotents, since $q_n$ is scalar, and the holomorphic functional calculus of such a sum is homotopic to the sum of the functional calculi; hence $D_p(e)=[h(I_p(e))]-[q_n]$ is also the relative class of $(pe,q_ne)$.
\end{remark}

Granting this, the quantitative $K$-homology classes $\nu_0(z)$ coming from Lipschitz $K$-theory pair with quantitative $K$-theory exactly as in \cite{GuoYu2023Covering}, while the pairing of Definition \ref{def_pairing} is defined on all of $K^0_\eps(E_r)$.

\subsection{Comparison with finite test sets}
The proof of Proposition \ref{prop_pairing_well_defined} used the whole unit ball $E_r$ as test set in an essential way: Lemma \ref{lem_pairing_commutator} controls the commutator of $p$ with every $e\in M_k(\C_u^*[X]_r)$, uniformly in $k$, because $E_r$ contains the diagonal signs and all shifts of length at most $r$. If $p$ is only known to almost commute with a fixed finite subset $F$ of $E_r$, the same uniform estimate is not available, and the formula for the pairing need not be invariant under the equivalence relation defining $K^0_\eps(F)$. This is the main difference between our definition and the controlled $KK$-theory of \cite{WYUCT_control_kk2}, which uses a filtration by finite subsets. 

The two approaches can nevertheless be compared. Let $\calX'\subseteq\calX$ consist of the pairs $(F,\eps)\in\calX$ with $F$ finite, equipped with the induced partial order. Note that $QK^*(X)$ can be identified with the inverse limit over $\bigcup_{r\geq 0}\calX_r$. Given a coherent family $(g_{(G,\delta)})$ indexed by $\bigcup_{r\geq 0}\calX_r$ and a pair $(F,\eps)\in\calX'$, approximate each element of the finite set $F$ in norm by an element of the unit ball of $\C_u^*[X]_r$ for some $r$, to obtain $(G,\delta)\in\calX_r$ with $(F,\eps)\le(G,\delta)$, and set $\Xi(g)_{(F,\eps)}=\phi^{(G,\delta)(F,\eps)}(g_{(G,\delta)})$. Since any two such choices $(G_1,\delta_1),(G_2,\delta_2)$ are dominated by $(G_1\cup G_2,\min(\delta_1,\delta_2))$, coherence of $g$ shows that this is independent of the choice, and we obtain a canonical homomorphism
\[
\Xi\colon QK^*(X)=\varprojlim_{r\to \infty} QK^*_r(X)\longrightarrow \varprojlim_{(F,\eps)\in \calX'} K^*_\eps(F).
\]
Finite subsets are not cofinal in $\calX$ when $X$ is infinite: every $E_r$ contains the unit ball of $\ell^\infty(X)$, whose projections $\chi_A$, $A\subseteq X$, are at mutual distance $1$, so no finite subset approximates $E_r$ to within a small norm error. Thus $\Xi$ involves a genuine passage from uniform control over the whole propagation ball to control over finite test sets, and we do not know whether it is an isomorphism.

\begin{conjecture}
    Let $X$ be a discrete proper metric space with bounded geometry which is $(m,C)$-chain geodesic. Then the canonical map $\Xi$ is an isomorphism. 
\end{conjecture}

\newpage
\paragraph*{AI disclosure.}
AI tools were used for language editing and feedback
on the exposition and consistency of the manuscript.
The paper's original results and proofs were developed by
the authors, who reviewed the revisions and take full
responsibility for the final manuscript.

\printbibliography
\vspace{2em}
\noindent{\small
\textsc{Hao Guo}, i14, Melbourne, Australia\\
\textit{Email address}: \href{mailto:hao.guo@i14.ai}{\texttt{hao.guo@i14.ai}}

\medskip
\noindent\textsc{Chenxi Wang}, Department of Mathematics, Texas A\&M University, College Station, Texas, USA\\
\textit{Email address}: \href{mailto:ccw@tamu.edu}{\texttt{ccw@tamu.edu}}

\medskip
\noindent\textsc{Guoliang Yu}, Department of Mathematics, Texas A\&M University, College Station, Texas, USA\\
\textit{Email address}: \href{mailto:guoliangyu@tamu.edu}{\texttt{guoliangyu@tamu.edu}}
}

\end{document}